\documentclass{amsart}
\usepackage{amsmath, amssymb, amsthm}
\usepackage{amscd}
\usepackage{graphicx} 
\usepackage{enumerate}

\title{On the heights of totally $p$-adic numbers}
  \author[Fili]{Paul Fili}
 \address{Department of Mathematics\\ University of Rochester, Rochester, NY 14627}
 \email{fili@math.rochester.edu}
\date{\today}

\usepackage[pdftitle={On the heights of totally $p$-adic numbers},pdfauthor={Fili},pdfstartview={}]{hyperref}

\usepackage[all]{xy}

\newtheorem{thm}{Theorem}
\newtheorem{conj}{Conjecture}

\newtheorem{cor}{Corollary}

\newtheorem*{thm*}{Theorem}
\newtheorem*{alg*}{Algorithm}
\newtheorem*{lemma*}{Lemma}

\theoremstyle{remark}

\newtheorem*{rmk*}{Remark}

\newtheorem*{notation*}{Notation}

\theoremstyle{definition}
\newtheorem{defn}{Definition}
\newtheorem*{defn*}{Definition}

\newcommand{\mybf}{\mathbb}

\newcommand{\bP}{\mybf{P}}
\newcommand{\bR}{\mybf{R}}

\newcommand{\bC}{\mybf{C}}

\newcommand{\bQ}{\mybf{Q}}
\newcommand{\bF}{\mybf{F}}

\newcommand{\cD}{\mathcal{D}}

\newcommand{\cX}{\mathcal{X}}
\newcommand{\al}{\alpha}

\providecommand{\abs}[1]{\lvert#1\rvert}
\providecommand{\norm}[1]{\lVert#1\rVert}

\newcommand{\ra}{\rightarrow}

\newcommand{\ep}{\epsilon}

\newcommand{\Qbar}{\overline{\mybf{Q}}}
\newcommand{\Kbar}{\overline{K}}

\def\talltareesidedbox#1{\setbox0=\hbox{$#1$}\dimen0=\wd0 \advance\dimen0 by3pt\rlap{\hbox{\vrule height10pt width.4pt
 depth2pt \kern-.4pt\vrule height10.4pt width\dimen0 depth-10pt\kern-.4pt \vrule height10pt width.4pt depth2pt}}
 \relax \hbox to\dimen0{\hss$#1$\hss}}
\def\tareesidedbox#1{\setbox0=\hbox{$#1$}\dimen0=\wd0 \advance\dimen0 by3pt\rlap{\hbox{\vrule height8pt width.4pt
 depth2pt \kern-.4pt\vrule height8.4pt width\dimen0 depth-8pt\kern-.4pt \vrule height8pt width.4pt depth2pt}}
\relax \hbox to\dimen0{\hss$#1$\hss}}
\def\shorttareesidedbox#1{\setbox0=\hbox{$#1$}\dimen0=\wd0 \advance\dimen0 by3pt\rlap{\hbox{\vrule height7pt width.4pt
 depth2pt \kern-.4pt\vrule height7.4pt width\dimen0 depth-7pt\kern-.4pt \vrule height7pt width.4pt depth2pt}}
 \relax \hbox to\dimen0{\hss$#1$\hss}}

\newcommand{\sP}{\mathsf{P}}
\newcommand{\sA}{\mathsf{A}}

\begin{document}

\begin{abstract}
 Bombieri and Zannier established lower and upper bounds for the limit infimum of the Weil height in fields of totally $p$-adic numbers and generalizations thereof. In this paper, we use potential theoretic techniques to generalize the upper bounds from their paper and, under the assumption of integrality, to improve slightly upon their bounds.
\end{abstract}

\maketitle

\section{Statement of Results}
Recall that an algebraic number is said to be \emph{totally $p$-adic} if its image lies in $\bQ_p$ for any embedding $\Qbar\hookrightarrow \bC_p$, where $\bC_p$ denotes the completion of an algebraic closure of $\bQ_p$. This is analogous to the usual definition of a totally real number, however, unlike $\bC/\bR$, the extension $\bC_p/\bQ_p$ is of infinite degree, so in fact we can make an even broader generalization:
\begin{defn}
Let $L_p/\bQ_p$ be a (finite) Galois extension for $p\leq \infty$ a rational prime. We say $\al\in\Qbar$ is \emph{totally $L_p$} if all Galois conjugates of $\al$ lie in $L_p\subset \bC_p$.
\end{defn}
More generally, when our objects are defined over an arbitrary number field $K$, we make the following definition:
\begin{defn}
Fix a base number field $K$, and let $S$ be a set of places of $K$. For each $v\in S$, we choose a Galois extension $L_v/K_v$. We say that $\al$ is \emph{totally $L_S/K$} if, for each $v\in S$, all of the $K$-Galois conjugates of $\al$ lie in $L_v$.
\end{defn}

\noindent Notice that $\al$ is totally $L_S/K$ if and only if the minimal polynomial for $\al$ over $K$ splits in $L_v$ for each $v\in S$. In our terminology, a number being totally real is equivalent to being totally $\bR$ (or totally $\bR/\bQ$), and being totally $p$-adic is equivalent to being totally $\bQ_p$ (totally $\bQ_p/\bQ$). Notice that the set of all totally $L_S/K$ numbers form a normal extension of $K$ (typically of infinite degree). 

Bombieri and Zannier \cite{BombieriZannierNote} studied the question of what the limit infimum of the Weil height was in these fields when the base field was assumed to be $\bQ$, and they proved the following:
\begin{thm*}[Bombieri and Zannier 2001]
Let $L/\bQ$ be a normal extension (possibly of infinite degree) and $S$ is the set of finite rational primes such that $L_p/\bQ_p$ is Galois (in particular, of finite degree), then
\[
 \liminf_{\al\in L} h(\al) \geq \frac{1}{2} \sum_{p\in S} \frac{\log p}{e_p(p^{f_p} + 1)}
\]
where $e_p$ and $f_p$ denote the ramification and inertial degrees of $L_p/\bQ_p$, respectively.

Further, in the case where $S=\{p_1,\ldots,p_n\}$ is a finite set of nonarchimedean rational primes with $L_{p_i}=\bQ_{p_i}$ for each $i$, if we let $L$ be the field of all totally $L_S/\bQ$ numbers, then we have
\[
 \liminf_{\al\in L} h(\al) \leq \sum_{i=1}^n \frac{\log p_i}{p_i - 1}.
\]
\end{thm*}

Smyth \cite{SmythTotRealI,SmythTotRealII} and Flammang \cite{Flammang} proved analogous results in the totally real case for the limit infimum of the height, however, they imposed the additional hypothesis that one consider only totally real integers.

This note has two aims: first, to prove a slightly sharper lower bound by adding the hypothesis that we consider only integers, and second, to generalize the upper bound from Bombieri and Zannier's paper by the aid of the Fekete-Szeg\H{o} theorem with splitting conditions, as formulated and proven by Rumely Rumely \cite{RumelyFeketeI,RumelyFeketeII}. Our results are the following:

\begin{thm}\label{thm:totpadic}
Fix a number field $K$, a set of nonarchimedean places $S$ of $K$, and a choice of Galois extension $L_v/K_v$ for each $v\in S$. Let $L$ be the field of all totally $L_S/K$ numbers, and $O_{L}$ denote its ring of integers. Then
 \[
 \liminf_{\al\in O_L} h(\al) \geq \frac{1}{2} \sum_{v\in S} N_v\cdot \frac{\log p_v}{e_v(q_v^{f_v} - 1)}
\]
where $N_v=[K_v:\bQ_v]/[K:\bQ]$, $e_v$ and $f_v$ denote the ramification and inertial degrees of $L_v/K_v$, respectively, and $q_v$ denotes the order of the residue field of $K_v$, and $p_v$ is rational prime above which $v$ lies.
\end{thm}
\noindent As Bombieri and Zannier remark in their note, if the above sum diverges (which may happen if we allow $S$ to be infinite), then it follows that the set of totally $L_S/K$ integers satisfies the Northcott property (there are only a finite number with height below any fixed constant).

If we restrict our attention to extensions normal over $\bQ$, as Bombieri and Zannier do, our result reduces to:
\begin{cor}
Fix a set of finite rational primes $S$ and a choice of Galois extension $L_p/\bQ_p$ for each $p\in S$. Let $L$ be the field of all totally $L_S$ numbers, and $O_{L}$ denote its ring of integers. Then
 \[
 \liminf_{\al\in O_L} h(\al) \geq \frac{1}{2} \sum_{p\in S} \frac{\log p}{e_p(p^{f_p} - 1)}
\]
where $e_p$ and $f_p$ denote the ramification and inertial degrees of $L_p/\bQ_p$, respectively.
\end{cor}

\begin{thm}\label{thm:p-adic-upper-bound}
 Fix a number field $K$, a finite set of nonarchimedean places $S$ of $K$, and a choice of Galois extension $L_v/K_v$ for each $v\in S$ with ramification degree $e_v$, inertial degree $f_v$, and residue field degree $q_v$ and characteristic $p_v$. Let $L$ be the field of all totally $L_S/K$ numbers. Then
\[
 \liminf_{\al \in L} h(\al) \leq \sum_{v\in S} N_v\cdot \frac{\log p_v}{e_v(q_v^{f_v}-1)}.
\]
 where $N_v=[K_v:\bQ_v]/[K:\bQ]$.
\end{thm}
\noindent Notice that this differs from our lower bound only by the factor of $1/2$.

One interesting thing to note in Theorem \ref{thm:totpadic} is that the shape of our bound comes directly from the formula for the logarithmic capacity of the ring of integers of $L_v$ at each place. This connection to potential theory gives an interesting explanation for the shape of the observed bound, and suggests that by determining minimal energy measures on $L_v$ one might further improve the bounds above. In fact, this connection inspires us to conjecture that in fact the upper bound in Theorem \ref{thm:p-adic-upper-bound} is sharp:

\begin{conj}
 Fix a number field $K$, a finite set of nonarchimedean places $S$ of $K$, and a choice of Galois extension $L_v/K_v$ for each $v\in S$ with ramification degree $e_v$, inertial degree $f_v$, and residue field of degree $q_v$ and characteristic $p_v$. Let $L$ be the field of all totally $L_S/K$ numbers and $O_L$ its ring of integers. Then
\[
 \liminf_{\al \in O_L} h(\al) = \sum_{v\in S} N_v\cdot \frac{\log p_v}{e_v(q_v^{f_v}-1)}.
\]
 where $N_v=[K_v:\bQ_v]/[K:\bQ]$. 
\end{conj}

It is still an interesting open question whether the limit infimum for all numbers can be in fact achieved with integers or not. In this direction, we will note only that the result of Bombieri and Zannier, when read as an equidistribution result, seems to indicate that totally $p$-adic points of low height should be distributing evenly in the residue classes of $\bP^1(\bF_p)$. This seems to suggest that the limit infimum over an entire field $L$ of the type constructed above might be smaller than that obtained by integers; however, all of the smallest known limit points, both for the height of totally real numbers (see \cite{SmythTotRealI}) and the height of totally $p$-adic numbers (as in Theorem \ref{thm:p-adic-upper-bound}) are in fact achieved by sequences of integers.

\section{Proofs}
\begin{proof}[Proof of Theorem \ref{thm:totpadic}]
We now proceed to prove Theorem \ref{thm:totpadic}. Recall that we have fixed a base number field $K$, and let $M_K$ denotes the places of $K$. First, let us note that if $S$ is infinite, we can take a limit over increasing finite subsets of $S$, and the general result will follow, so we may as well assume that $S$ is finite in our proof. For convenience let 
\[
 N_v = \frac{[K_v:\bQ_v]}{[K:\bQ]}.
\]
We recall Baker's reformulation of Mahler's inequality from \cite{BakerAverages}, namely, if $\al$ is an algebraic number and $\al_1,\ldots,\al_n$ denote its $K$-Galois conjugates, then
\begin{equation}\label{eqn:baker-mahler}
 \frac{1}{n(n-1)} \sum_{i\neq j} g_v(x,y) \geq \begin{cases}\displaystyle
                                                    -N_v \cdot \frac{\log n}{n-1} & \text{if }v\mid \infty,\\
                                                    0 & \text{otherwise}
                                                   \end{cases}
\end{equation}
where
\[
 g_p(x,y) = \log^+\abs{x}_v + \log^+\abs{y}_v - \log \abs{x-y}_v,
\]
and our absolute values are normalized so that $\abs{\cdot}_v = \norm{\cdot}_v^{N_v}$ where $\norm{\cdot}_v$ extends the usual absolute of $\bQ$ over which it lies (thus the absolute values $\abs{\cdot}_v$ satisfy the product formula, and $h(\al)= \sum_v \log^+\abs{\al}_v$). 
We assume in our theorem that $\al$ is integral. The key role is played by the discriminant here. At the place $v\in M_K$ ($p$ here being the rational prime over which the place $v$ lies), all of the conjugates of $\al$ lie in ring of integers $O_{L_v}$. By \cite[Example 4.1.24]{RumelyBook}, we have that the $v$-adic logarithmic capacity of $O_{L_v}$ with respect to the point $\infty$, which we will denote $\gamma_{\infty,v}(O_{L_v})$, satisfies
\[
 \log \gamma_{\infty,v}(O_{L_v}) = -N_v\cdot \frac{\log p_v}{e_v(q_v^{f_v}-1)} < 0,
\]
where $q_v$ denotes the order of the residue field of $O_{L_v}$ and $p_v$ its characteristic.\footnote{In the text \cite{RumelyBook} the $v$-adic absolute value $\abs{\cdot}_v$ which is used to compute the capacity is normalized to agree with the modulus with respect to the additive Haar measure, so in the text a $\log q_v$ appears in the numerator instead of $\log p_v$. We also wish to draw the reader's attention to the fact that the ring of integers $O_{L_v}$ is `strictly smaller' than the closed unit disc in $\sP^1(\bC_v)$, which has logarithmic capacity $\log \gamma_{\infty,v}(\cD_v(0,1))=0$.} We can also view this as the transfinite diameter of $O_{L_v}$; more specifically, if we adopt the notation of \cite[Ch. 6]{BakerRumelyBook} and let
\[
 \log d_n(O_{L_v})_\infty = \sup_{z_1,\ldots,z_n\in O_{L_v}} \frac{1}{n(n-1)} \log \abs{z_i-z_j}_v,
\]
then by \cite[Lemma 6.21 and Theorem 6.23]{BakerRumelyBook} it follows that\footnote{Note that here for convenience our absolute values are normalized so that we have the same normalization factors as in computing the absolute logarithmic Weil height.}
\[
 \lim_{n\ra\infty} \log d_n(O_{L_v})_\infty = \log \gamma_{\infty,v}(O_{L_v}).
\]

Let $\{\al^{(k)}\}_{k=1}^\infty$ denote a sequence of algebraic numbers such that
\[
 \lim_{k\ra\infty} h(\al^{(k)}) = \liminf_{\al\in O_L} h(\al),
\]
and let $n_k$ denote the number of $K$-Galois conjugates of $\al^{(k)}$, which we denote $\al^{(k)}_1,\ldots, \al^{(k)}_{n_k}$. Since the $\al^{(k)}$ have bounded height, it follows from Northcott's theorem that $n_k\ra\infty$ as $k\ra\infty$. By definition of $d_n(O_{L_v})_\infty $, we have
\[
 \frac{1}{n_k(n_k-1)}\sum_{i\neq j} \log \abs{\al^{(k)}_i-\al^{(k)}_j}_v\leq \log d_{n_k}(O_{L_v})_\infty 
\]
for any totally $L_S/K$ integer and $v\in S$, so it follows that 
\[
 \limsup_{k\ra\infty} \sum_{v\in S} \frac{1}{n_k(n_k-1)}\sum_{i\neq j} \log \abs{\al^{(k)}_i-\al^{(k)}_j}_v\leq \sum_{v\in S} \log \gamma_{\infty,v}(O_{L_v}).
\]
Therefore, by the product formula, using the assumption of integrality,
\[
 \liminf_{k\ra\infty} \sum_{\substack{w\in M_K\\ w\mid \infty}}
 \frac{1}{n_k(n_k-1)}\sum_{i\neq j} \log \abs{\al^{(k)}_i-\al^{(k)}_j}_w
 \geq -\sum_{v\in S} \log \gamma_{\infty,v}(O_{L_v}).
\]
It now follows from \eqref{eqn:baker-mahler} that
\begin{multline*}
 \liminf_{k\ra\infty} 2h(\al^{(k)})\geq \liminf_{k\ra\infty} \sum_{w\mid\infty} \log^+\abs{\al^{(k)}}_w\\
 \geq -\sum_{v\in S} \log \gamma_{\infty,v}(O_{L_v}) = \sum_{v\in S} N_v\cdot \frac{\log p_v}{e_v(q_v^{f_v}-1)}.
\end{multline*}
Upon dividing each side by $2$, the result follows.
\end{proof}
\begin{proof}[Proof of Theorem \ref{thm:p-adic-upper-bound}]
 We will prove the result by applying the adelic Fekete-Szeg\H{o} theorem with splitting conditions \cite[Theorem 2.1]{RumelyFeketeII} (see also \cite{RumelyFeketeI} as well as \cite[Theorem 6.27]{BakerRumelyBook} for an interpretation in terms of the Berkovich topology). For each $v\in S$, let $O_{L_v}\subset L_v$ denote the ring of integers in $L_v$. Notice that $O_{L_v}\subset \sP^1(\bC_p)$ is a compact set and let $\cD_v(0,1)\subset \sP^1(\bC_v)$ be the usual Berkovich unit disc. Let $\mathbb{E}\subset \prod_{v\in M_K} \sA^1(\bC_v)$ be the adelic Berkovich set given by
\[
 \mathbb{E} = \prod_{v\in M_K} E_v\quad\text{where}\quad E_v =
\begin{cases}
 O_{L_v}& \text{for }v\in S\\
 \cD_v(0,1) &\text{for }v\not\in S,\\
 \cD\bigg(0,\displaystyle \prod_{v\in S} q_v^{N_v/e_v(q_v^{f_v}-1)}\bigg) & \text{for }v\mid \infty.
\end{cases}
\]
By \cite[Example 4.1.24]{RumelyBook} we have normalized $v$-adic logarithmic capacity
\[
 \log \gamma_{\infty,v}(O_{L_v}) = -N_v\cdot \frac{\log p_v}{e_v(q_v^{f_v}-1)}\quad\text{for each}\quad v\in S,
\]
where $\gamma_\infty$ denotes the capacity relative to $\cX = \{\infty\}$ in the notation of \cite{RumelyFeketeII}. At the infinite places, we have
\[
 \sum_{w\mid\infty} \log \gamma_{\infty,w}(\cD(0,\prod_{v\in S}q_v^{N_v/e_v(q_v^{f_v}-1)})) = \sum_{v\in S} N_v\cdot \frac{\log p_v}{ e_p(p^{f_p}-1)},
\]
 so that the adelic capacity, that is, the product of all of the normalized capacities, is
\[
 \gamma_\infty(\mathbb{E}) = \prod_{v\in M_K} \gamma_{\infty,v}(E_v) = 1.
\]
Further, by our assumptions, $\mathbb{E}$ is stable under the action of the continuous Galois automorphisms of $\bC_v/K_v$ at each place. For $\ep>0$, we consider the Berkovich adelic neighborhood
\[
 \mathbb{U} = \prod_{w\in M_K} U_w \quad\text{where}\quad U_w =
\begin{cases}
 E_w &\text{for }w\nmid \infty \\
 \cD\bigg(0,\displaystyle e^{\ep} \prod_{v\in S} q_v^{N_v/e_v(q_v^{f_v}-1)}\bigg)^{-} & \text{for }w\mid\infty,
\end{cases}
\]
where $\cD(a,r)^{-}$ denotes the usual open disc centered at $a$ of radius $r$. Then we can apply the adelic Fekete-Szeg\H{o} theorem with splitting conditions as formulated in \cite[Theorem 2.1]{RumelyFeketeII}, there are infinitely many algebraic numbers in $\Kbar$ all of whose conjugates lie in $\mathbb{U}$ with the additional condition that all $\bC_v/K_v$ conjugates lie in $U_v= E_v=O_{L_v}$ for each $v\in S$. Clearly, the only contribution to the height of such numbers comes from the archimedean places, and as a result, we have
\[
 h(\al) \leq \sum_{v\in S} N_v\cdot \frac{\log p_v}{e_v(q_v^{f_v}-1)} + \ep.
\]
Thus by letting $\ep \ra 0$ we can generate a sequence of numbers (in fact, integers) with the limit infinimum of the height bounded by the desired constant.
\end{proof}

\bibliographystyle{abbrv} 
\bibliography{bib}        

\end{document}